\documentclass[a4paper,12pt]{article}
\usepackage{a4wide}
\usepackage{amsmath}
\usepackage{amssymb}
\usepackage{amsthm}
\usepackage{latexsym}
\usepackage{graphicx}
\usepackage[english]{babel}
\usepackage{makeidx}
\usepackage[mathlines]{lineno}

\newtheorem{obs} [subsection]{Remark}
\newtheorem{exm} [subsection]{Example}

\newtheorem{prop}[subsection]{Proposition}

\newtheorem{teor}[subsection]{Theorem}
\newtheorem{lema}[subsection]{Lemma}

\def\sdepth{\operatorname{sdepth}}
\def\depth{\operatorname{depth}}
\def\supp{\operatorname{supp}}
\def\deg{\operatorname{deg}}

\def\Ass{\operatorname{Ass}}

\begin{document}
\selectlanguage{english}
\frenchspacing

\numberwithin{equation}{section}

\title{Depth and Stanley depth of powers of the path ideal of a cycle graph. II}
\author{Silviu B\u al\u anescu$^1$ and Mircea Cimpoea\c s$^2$}
\date{}

\maketitle

\footnotetext[1]{ \emph{Silviu B\u al\u anescu}, University Politehnica of Bucharest, Faculty of
Applied Sciences, 
Bucharest, 060042, E-mail: silviu.balanescu@stud.fsa.upb.ro}
\footnotetext[2]{ \emph{Mircea Cimpoea\c s}, University Politehnica of Bucharest, Faculty of
Applied Sciences, 
Bucharest, 060042, Romania and Simion Stoilow Institute of Mathematics, Research unit 5, P.O.Box 1-764,
Bucharest 014700, Romania, E-mail: mircea.cimpoeas@upb.ro,\;mircea.cimpoeas@imar.ro}

\begin{abstract}
Let $J_{n,m}:=(x_1x_2\cdots x_m,\;  x_2x_3\cdots x_{m+1},\; \ldots,\; x_{n-m+1}\cdots x_n,\; x_{n-m+2}\cdots x_nx_1, \linebreak
\ldots, x_nx_1\cdots x_{m-1})$ be the $m$-path ideal of the cycle graph of length $n$, in the ring of polynomials $S=K[x_1,\ldots,x_n]$.
As a continuation of our previous paper \cite{lucrare2}, 
we prove several new results regarding $\depth(S/J_{n,m}^t)$ and $\sdepth(S/J_{n,m}^t)$, 
where $t\geq 1$.


\end{abstract}

\section*{Introduction}

Let $n\geq 1$ be an integer. We denote $[n]=\{1,2,\ldots,n\}$. 
Let $K$ be a field and $S=K[x_1,\ldots,x_n]$ the polynomial ring in $n$ variables over $K$.
Given a simple graph $G=(V,E)$, on the vertex set $V=[n]$ with the edge set $E$, the edge ideal associated to $G$ is
$$I(G)=(x_ix_j\;:\;\{i,j\}\in E)\subset S.$$
Note that $I(G)$ is a monomial ideal generated in degree $2$. The study of the algebraic properties of $I(G)$ is a 
well established topic in combinatorial commutative algebra. Conca and De Negri generalized the definition of an edge ideal 
and first introduced the notion of a $m$-path ideal in \cite{conca}, that is
$$I_m(G)=(x_{i_1}x_{i_2}\cdots x_{i_m}\;:\;i_1i_2\cdots i_m\text{ is a path in }G)\subset S.$$
In the recent years, several algebraic and combinatorial properties of path ideals have been studied. However, the study
of powers path ideals is in its infancy. 

The path graph of length $n-1$ is $P_n=(V(P_n),E(P_n))$, where 
$$V(P_n)=[n]\text{ and }E(P_n)=\{\{1,2\},\{2,3\},\ldots,\{n-1,n\}\}.$$
Let $1\leq m\leq n$. One can easily check that the \emph{$m$-path ideal of the path graph} of length $n$ is
$$I_{n,m}=(x_1x_2\cdots x_m,\;  x_2x_3\cdots x_{m+1},\; \ldots,\; x_{n-m+1}\cdots x_n)\subset S.$$
The cycle graph of length $n$ is $C_n=(V(C_n),E(C_n))$, where 
$$V(C_n)=[n]\text{ and }E(C_n)=E(P_n)\cup\{\{1,n\}\}.$$
Also, for $2\leq m < n$, the \emph{$m$-path ideal of the cycle graph} of length $n$ is
$$J_{n,m}=I_{n,m}+(x_{n-m+2}\cdots x_nx_1,x_{n-m+3}\cdots x_nx_1x_2,\ldots,x_nx_1\cdots x_{m-1}).$$

In \cite{lucrare1} we studied the depth and Stanley depth of $S/I_{n,m}^t$, where $t\geq 1$. 
Also, in \cite{lucrare2} we studied the depth and Stanley depth of $S/J_{n,m}^t$, where $t\geq 1$.
Following \cite{lucrare2}, the aim of our paper is to further investigate 
the depth and Stanley depth of the quotient rings associated 
to powers of the $m$-path ideal of a cycle.

In Theorem \ref{t1}, we reprove and also extend some results from \cite{lucrare2}. As the new results,
we show that for any $n\geq 5$, we have:
\begin{align*}
& \sdepth(S/J_{n,n-2}^t),\depth(S/J_{n,n-2}^t)>0\text{ if }n\text{ is odd, and }t < \frac{n-1}{2},\\
& \sdepth(S/J_{n,n-2}^t), \depth(S/J_{n,n-2}^t)>0\text{ if }n\text{ is even, and }t \geq 1,\\
& \frac{n}{2}\geq \sdepth(S/J_{n,n-2}^t)\geq \depth(S/J_{n,n-2}^t)=1\text{ if }n\text{ is even, and }t \geq n-1.
\end{align*}
In Theorem \ref{t3}, we show that if $n\geq 2m+1$ then 
$$\depth(S/J_{n,m}^t) \leq \depth(S/I_{n,m}^t),$$
which improves the upper bound for $\depth(S/J_{n,m}^t)$ given in \cite[Theorem 2.10]{lucrare2}.

Finally, in Section \ref{s3} we make some small steps in tackling the problem of computing $\depth(S/J_{n,m}^t)$
for any $t\geq 1$, see Proposition \ref{obsy} and Proposition \ref{obsy2}. Also, 
 we illustrate in two examples the difficulties which appear, see Example \ref{exem1} and Example \ref{exem2}.



\section{Preliminaries}\label{s2}

First, we recall the well known Depth Lemma, see for instance \cite[Lemma 2.3.9]{real}. 

\begin{lema}\label{l11}(Depth Lemma)
If $0 \rightarrow U \rightarrow M \rightarrow N \rightarrow 0$ is a short exact sequence of modules over a local ring $S$, or a Noetherian graded ring with $S_0$ local, then
\begin{enumerate}
\item[(1)] $\depth M \geq \min\{\depth N,\depth U\}$.
\item[(2)] $\depth U \geq \min\{\depth M,\depth N +1\}$.
\item[(3)] $\depth N \geq \min\{\depth U-1,\depth M\}$.
\end{enumerate}
\end{lema}

The following result, which will be used later on, is an easy application of the Depth Lemma:

\begin{lema}\label{liema}
Let $d\geq 1$ and $Z_1\cup Z_2\cup \cdots \cup Z_d=\{x_1,\ldots,x_n\}$ be a partition, i.e. $|Z_i|>0$ and $Z_i\cap Z_j=\emptyset$ for all $i\neq j$.
Let $P_i=(Z_i)\subset S$ for $1\leq i\leq d$ and $U:=P_1\cap \cdots \cap P_d$. Then $\depth(S/U)=d-1$.
\end{lema}

Now, we briefly recall the definition of the Stanley depth invariant, for a quotient of monomial ideals.

Let $0\subset I\subsetneq J\subset S$ be two monomial ideals and $M=J/I$. A \emph{Stanley decomposition} of $M$ is the decomposition of $M$ as a direct sum 
$\mathcal D: M = \bigoplus_{i=1}^r m_i K[Z_i]$ of $\mathbb Z^n$-graded $K$-vector spaces, where $m_i\in S$ are monomials and $Z_i\subset\{x_1,\ldots,x_n\}$.
We define $\sdepth(\mathcal D)=\min_{i=1,\ldots,r} |Z_i|$ and $\sdepth(M)=\max\{\sdepth(\mathcal D)|\;\mathcal D$ is a Stanley decomposition of $M\}$. The number 
$\sdepth(M)$ is called the \emph{Stanley depth} of $M$. 

Herzog, Vladoiu and Zheng show in \cite{hvz} that $\sdepth(M)$ can be computed in a finite number of steps.
We say that $M$ satisfies the Stanley inequality, if 
$$\sdepth(M)\geq \depth(M).$$
Stanley \cite{stan} conjectured that any quotient of monomial ideals $M=J/I$ satisfies the Stanley inequality, a conjecture which proves to be false
in general for $M=J/I$, where $I\neq 0$; see Duval et al. \cite{duval}.

The explicit computation of the Stanley depth it is a difficult task, both from a theoretical and practical point of vue.
Also, although the Stanley conjecture was disproved in the most general set up, 
it is interesting to find large classes of quotients of monomial ideals which satisfy the Stanley inequality. 

In \cite{asia}, Asia Rauf proved the analog of Lemma \ref{l11} for $\sdepth$:

\begin{lema}\label{asia}
If $0 \rightarrow U \rightarrow M \rightarrow N \rightarrow 0$ is a short exact sequence of $\mathbb Z^n$-graded $S$-modules, then
$$ \sdepth(M) \geq \min\{\sdepth(U),\sdepth(N) \}.$$
\end{lema}



We also recall the following well known results. See for instance \cite[Corollary 1.3]{asia}, \cite[Proposition 2.7]{mirci},
\cite[Theorem 1.1]{mir}, \cite[Lemma 3.6]{hvz} and \cite[Corollary 3.3]{asia}.

\begin{lema}\label{lem}
Let $I\subset S$ be a monomial ideal and let $u\in S$ a monomial such that $u\notin I$. Then
\begin{enumerate}
\item[(1)] $\sdepth(S/(I:u))\geq \sdepth(S/I)$.
\item[(2)] $\depth(S/(I:u))\geq \depth(S/I)$.
\end{enumerate}
\end{lema}

\begin{lema}\label{lemm}
Let $I\subset S$ be a monomial ideal and let $u\in S$ a monomial such that $I=u(I:u)$. Then
  \begin{enumerate}
  \item[(1)] $\sdepth(S/(I:u))=\sdepth(S/I)$.
	\item[(2)] $\depth(S/(I:u))=\depth(S/I)$.
  \end{enumerate}
\end{lema}	
	
	
\begin{lema}\label{lhvz}
Let $I\subset S$ be a monomial ideal and $S'=S[x_{n+1}]$. Then
\begin{enumerate}
\item[(1)] $\sdepth_{S'}(S'/IS')=\sdepth_S(S/I)+1$, 
\item[(2)] $\depth_{S'}(S'/IS')=\depth_S(S/I)+1$.
\end{enumerate}
\end{lema}

\begin{lema}\label{lem7}
Let $I\subset S$ be a monomial ideal. Then the following assertions are equivalent:
\begin{enumerate}
\item[(1)] $\mathfrak m:=(x_1,\ldots,x_n)\in \Ass(S/I)$.
\item[(2)] $\depth(S/I)=0$.
\item[(3)] $\sdepth(S/I)=0$.
\end{enumerate}
\end{lema}

We will use later on the following result from \cite{iran}:

\begin{teor}\label{teo-iran}(see \cite[Theorem 2.11]{iran})

Let $1\leq p\leq n-1$, $S'=K[x_1,\ldots,x_p]$ and $S''=K[x_{p+1},\ldots,x_n]$. Let $I\subset S'$ be a monomial ideal
and let $L\subset S''$ be a complete intersection monomial ideal. We denote by $I+L$, the ideal $IS+LS$ of 
$S=S'\otimes_K S''=K[x_1,\ldots,x_n]$. Then, for all $t\geq 1$, we have that
\begin{enumerate}
\item[(1)] $\depth(S/(I+L)^t)=\min\limits_{1\leq i\leq t}\{ \depth_{S'}(S'/I^i) \}+\dim(S''/L)$.
\item[(2)] $p+\dim(S''/L)\geq \sdepth(S/(I+L)^t)\geq \min\limits_{1\leq i\leq t}\{ \sdepth_{S'}(S'/I^i)\}+\dim(S''/L)$.
\end{enumerate}
\end{teor}

Let $1\leq m \leq n$ be two integers. The $m$-path ideal of the path graph $P_n$ is
$$I_{n,m}=(x_1\cdots x_m,\; x_2\cdots x_{m+1},\; \ldots,\; x_{n-m+1}\cdots x_n)\subset S.$$
We denote
$$\varphi(n,m,t):= \begin{cases} n -t+2 -
 \left\lfloor \frac{n-t+2}{m+1} \right\rfloor - \left\lceil \frac{n-t+2}{m+1} \right\rceil, & t \leq n+1-m
 \\ m-1,& t > n+1-m \end{cases}.$$
We recall the main result of \cite{lucrare1}:

\begin{teor}(See \cite[Theorem 2.6]{lucrare1})\label{depth}
With the above notations, we have that
\begin{enumerate}
\item[(1)] $\sdepth(S/I_{n,m}^t)\geq \depth(S/I_{n,m}^t)=\varphi(n,m,t),\text{ for all }t\geq 1$.
\item[(2)] $\sdepth(S/I_{n,m}^t)\leq \sdepth(S/I_{n,m})=\varphi(n,m,1)$.
\end{enumerate}
\end{teor}


Let $2\leq m < n$ be two integers. The $m$-path ideal of the cycle graph $C_n$ is
$$J_{n,m}=I_{n,m}+(x_{n-m+2}\cdots x_nx_1,\; x_{n-m+3}\cdots x_nx_1x_2,\; \ldots,\; x_nx_1\cdots x_{m-1}).$$
Let $d=\gcd(n,m)$ and let $t_0:=t_0(n,m)$ be the maximal integer such that $t_0\leq n-1$ and there exists a positive
integer $\alpha$ such that 
$$mt_0 = \alpha n + d.$$
Let $t\geq t_0$ be an integer.
Let $w=(x_1x_2\cdots x_n)^{\alpha}$, $w_t=w\cdot (x_1\cdots x_m)^{t-t_0}$, $r:=\frac{n}{d}$ and $s:=\frac{m}{d}$.
If $d>1$, we consider the ideal
$$U_{n,d}=(x_1,x_{d+1},\cdots,x_{d(r-1)+1})\cap (x_2,x_{d+2},\cdots,x_{d(r-1)+2})\cap \cdots \cap (x_d,x_{2d},\ldots,x_{rd}).$$
We recall the following results from \cite{lucrare2}:

\begin{lema}\label{lucky}(\cite[Lemma]{lucrare2})
With the above notations, we have:
\begin{enumerate}
\item[(1)] If $d=1$ then $(J_{n,m}^{t}:w_t)=\mathfrak m$ for all $t\geq t_0$.
\item[(2)] If $d>1$ then $(J_{n,m}^{t}:w_t)=U_{n,d}$ for all $t\geq t_0$.
\end{enumerate}
\end{lema}

We also recall the following result of \cite{lucrare2}:

\begin{teor}(\cite[Theorem 2.12]{lucrare2})\label{t212}
With the above notations, we have that $$\depth(S/J_{n,m}^t)\leq \varphi(n-1,m,t)+1,\text{ for all }t\geq 1.$$
\end{teor}

\section{Main results}

In the beginning of this section, we prove the following lemma:

\begin{lema}\label{l1}
We have that:
\begin{enumerate}
\item[(1)] $\mathfrak m\in \Ass(S/J_{n,n-1}^t)$, for all $n\geq 2$ and $t\geq n-1$.
\item[(2)] If $n\geq 3$ is odd then $\mathfrak m\notin \Ass(S/J_{n,n-2}^t)$, for all $t<\frac{n-1}{2}$.
\item[(3)] If $n\geq 3$ is odd then $\mathfrak m\in \Ass(S/J_{n,n-2}^t)$, for all $t\geq \frac{n-1}{2}$.
\item[(4)] If $n\geq 3$ is even then $\mathfrak m\notin \Ass(S/J_{n,n-2}^t)$, for all $t\geq 1$.
\end{enumerate}
\end{lema}

\begin{proof}
(1) The result follows from Lemma \ref{lucky}(1). However, we present here a new proof:
Let $w_t:=x_1^{t-1}\cdots x_{n-1}^{t-1}x_n^{n-2}$. Note that $J_{n,n-1}^t$ is 
minimally generated by monomials of degree $(n-1)t$, while $\deg(w_t)=(n-1)t-1$. Thus, $w_t\notin J_{n,n-1}^t$.
We claim that
\begin{equation}\label{cleim}
(J_{n,n-1}^t:w_t)=\mathfrak m.
\end{equation}
Since  $(x_1\cdots x_{n-1})^{t-n+1}\in J_{n,n-1}^{t-n+1}$, $w_t=(x_1\cdots x_{n-1})^{t-n+1}w_{n-1}$ and $w_t\notin J_{n,n-1}^t$ it follows that
$$ (J_{n,n-1}^{n-1}:w_{n-1}) \subseteq (J_{n,n-1}^t:w_t) \subsetneq S .$$
Therefore, as $\mathfrak m$ is maximal, it is enough to prove \eqref{cleim} for $t=n-1$.

Note that $G(J_{n,n-1})=\{u_1,\ldots,u_n\}$, where $u_i=\prod_{j\neq i}x_j$ for all $1\leq j\leq n$. It
is easy to see that:
$$ x_jw_{n-1} = \prod_{k\neq j} u_k \in J_{n,n-1}^{n-1}\text{ for all }1\leq j\leq n,$$
hence \eqref{cleim} is true.

(2) Assume by contradiction that $\mathfrak m \in \Ass(S/J_{n,n-2}^t)$. Then there exists a monomial $w\in S$ with $w\notin J_{n,n-2}^t$ such that 
		$(J_{n,n-2}^t :w) = \mathfrak m$. By degree reason, $\deg(w)\geq t(n-2)-1$. Without any loss of generality, we may assume that
		$w=x_1^{a_1}\cdots x_n^{a_n}$ with $a_1\geq a_2\geq \cdots \geq a_n$. Then, we deduce that $a_1\geq t$. Since $x_1f\in J_{n,n-2}^t$ it
		follows that $w\in J_{n,n-2}^t$, a contradiction.

(3) The result follows from Lemma \ref{lucky}(1), since $\gcd(n,n-2)=1$ and $t_0(n,n-2)=\frac{n-1}{2}$ for $n$ odd.

(4) First, note that $d=\gcd(n,n-2)=2$.
		Assume by contradiction that there exists a monomial $w\in S$ with $w\notin J_{n,n-2}^t$ such that 
		$(J_{n,n-2}^t :w) = \mathfrak m$. It follows that $x_jw \in J_{n,n-2}^t$ for all $1\leq j\leq n$.
		Since $x_1w\in J_{n,n-2}^t$ and $w\notin J_{n,n-2}^{t}$ it follows that $w=u_1\cdots u_{t-1}v$,
		where $u_j\in G(J_{n,n-2})$, $v\notin J_{n,n-2}$ and $x_1v \in J_{n,n-2}$. 
		This implies
		\begin{enumerate}
		\item[(i)] $\supp(v)=\{x_1,x_2,\ldots,x_n\}\setminus \{x_1,x_j,x_{j+1}\}$, where $2\leq j\leq n-1$, or
		\item[(ii)] $\supp(v)=\{x_1,x_2,\ldots,x_n\}\setminus \{x_1,x_j\}$, where $3\leq j\leq n-1$.
		\end{enumerate}
    Note that, if $x_2(u_1\cdots u_{t-1})=x_{\ell}(u'_1\cdots u'_{t-1})$ for some $u'_j \in G(J_{n,n-2})$, then
		it follows that $\ell$ is even, since $u_j$'s and $u'_j$'s are products of $\frac{n-2}{2}$ variables with
		odd indices and $\frac{n-2}{2}$ variables with even indices. Therefore, if $x_2w\in J_{n,n-2}^t$ then
		$x_{\ell}v\in J_{n,n-2}$ for some even index $\ell$. 
		
		In the case (i), it follows that $\supp(x_{\ell}v)=\{x_1,x_2,\ldots,x_n\}\setminus \{x_1,x_j\}$ with $j\neq 1$ and $j$ odd. But this contradicts
		the fact that $x_{\ell}v\in J_{n,n-2}$. 
		
		In the case (ii), if $j$ is odd, then $\supp(x_{\ell}v)=\{x_1,x_2,\ldots,x_n\}\setminus\{x_1,x_j\}$ and
		again we get a contradiction. It follows that 
		\begin{equation}\label{kisu}
		\supp(v)=\{x_1,x_2,\ldots,x_n\}\setminus\{x_1,x_{2k}\},\text{ where }2\leq k\leq \frac{n-2}{2}.
		\end{equation}
		We
		claim that $w\in J_{n-2,2}^t$. Indeed, since $x_2w \in J_{n-2,2}^t$, from \eqref{kisu} it follows
		that there exist $u'_j\in G(J_{n,n-2})$ with $1\leq j\leq t-1$ such that $x_2(u_1\cdots u_{t-1})=x_{2k}(u'_1\cdots u'_{t-1})$.
		Therefore, 
		$$w=(u_1\cdots u_{t-1})v=x_2(u_1\cdots u_{t-1})\frac{v}{x_2}=x_{2k}(u'_1\cdots u'_{t-1})\frac{vx_{2k}}{x_2}.$$
		Denoting $v'=v\cdot\frac{x_{2k}}{x_2}$, it is easy to see that $\supp(v')\subset \{x_3,\ldots,x_n\}$. Hence, $v'\in J_{n,n-2}$
		and thus $w\in J_{n,n-2}^t$, a contradiction.		
\end{proof}		

		
\begin{teor}\label{t1}
We have that:
\begin{enumerate}
\item[(1)] $\sdepth(S/J_{n,n-1}^t)=\depth(S/J_{n,n-1}^t)=0$, for all $n\geq 2$ and $t\geq n-1$.
\item[(2)] If $n\geq 3$ is odd then $\sdepth(S/J_{n,n-2}^t)=\depth(S/J_{n,n-2}^t)=0$, for all $t\geq \frac{n-1}{2}$.
\item[(3)] If $n\geq 5$ is odd then $\sdepth(S/J_{n,n-2}^t),\depth(S/J_{n,n-2}^t)>0$, for all $t < \frac{n-1}{2}$.
\item[(4)] If $n\geq 5$ is even then $\sdepth(S/J_{n,n-2}^t), \depth(S/J_{n,n-2}^t)>0$, for all $t\geq 1$.
\item[(5)] If $n\geq 5$ is even then $\frac{n}{2}\geq \sdepth(S/J_{n,n-2}^t)\geq \depth(S/J_{n,n-2}^t)=1$, for all $t\geq n-1$.
\end{enumerate}
\end{teor}

\begin{proof}
(1), (2), (3) and (4) follows immediately from Lemma \ref{l1} and Lemma \ref{lem7}.

(5) It follows from (4) and \cite[Corollary 2.8]{lucrare2}.
\end{proof}

\begin{obs}\rm
Note that (1) from Theorem \ref{t1} was proved in \cite[Theorem 3.1]{lucrare2}
and (2) from Theorem \ref{t1} was proved in \cite[Corollary 2.8(1)]{lucrare2}.
The results (3), (4) and (5) from Theorem \ref{t1} are new.
\end{obs}

\begin{lema}\label{inmt2}
Let $n>m\geq 2$ and $t\geq 1$ be some integers. Then
$$(J_{n,m}^t:(x_{n-m+1}x_{n-m+2}\cdots x_{n-1})^t) = \begin{cases} (x_{n-m},x_n)^t,& n\leq 2m \\
             (I_{n-m-1,m}S+(x_{n-m},x_n))^t,& n\geq 2m+1 \end{cases}.$$
\end{lema}

\begin{proof}
It is easy to check that
\begin{equation}\label{eco1}
(J_{n,m}:(x_{n-m+1} \cdots x_{n-1})) = \begin{cases} (x_{n-m},x_n),& n\leq 2m \\ I_{n-m-1,m}S+(x_{n-m},x_n),& n\geq 2m+1 \end{cases}.
\end{equation}
Since $(J_{n,m}:(x_{n-m+1} \cdots x_{n-1}))^t \subset (J_{n,m}^t:(x_{n-m+1} \cdots x_{n-1})^t)$,  in order to
complete the proof, by \eqref{eco1}, it suffices to show that for any monomial $v$ with $(x_{n-m+1} \cdots x_{n-1})^t v \in J_{n,m}^t$,
there exists some monomials $v_1,\ldots,v_t\in \begin{cases} (x_{n-m},x_n),& n\leq 2m \\ I_{n-m-1,m}S+(x_{n-m},x_n),& n\geq 2m+1 \end{cases}$, such that $v_1\cdots v_t \mid v$. 

Indeed, let $v$ a monomial as above. Let $a=\deg_{x_{n-m}}(v)$ and $b=\deg_{x_n}(v)$. If $a\geq t$ then we can choose 
$v_1=\cdots=v_t=x_{n-m}$ and 
we are done. Also, if $a<t$ and $a+b\geq t$, we can choose $v_1=\cdots=v_a=x_{n-m}$, $v_{a+1}=\cdots=v_t=x_n$ and we are also done. 

Now, assume $a+b<t$. Let $v_1=\cdots=v_a=x_{n-m}$ and $v_{a+1}=\cdots=v_{a+b}=x_n$. Also, let $v'=\frac{v}{x_m^a x_n^b}$. 
Since $(x_{n-m+1} \cdots x_{n-1})^t v \in J_{n,m}^t$, there are $g_1,\ldots,g_t \in G(J_{n,m})$ such that 
$g_1\cdots g_t\mid (x_{n-m+1} \cdots x_{n-1})^t v$. It is clear that at most $a+b$ of the monomials $g_1,\ldots,g_t$ are divisible
by $x_{n-m}$ or $x_n$. Hence, there are $t-a-b$ such monomials, let's say $g_{a+b+1},\ldots,g_t$ which are not divisible neither by $x_{n-m}$,
neither by $x_n$. In particular, it follows that $$g_{a+b+1},\ldots,g_t\in K[x_1,\ldots,x_{n-m-1},x_{n-m+1},\ldots,x_{n-1}],$$ which lead to a contradiction if $n\leq 2m$. On the other hand, if $n\geq 2m+1$, then we get $g_{a+b+1},\ldots,g_t \in G(I_{n-m-1,m})$. We let $v_{a+b+1}=g_{a+b+1},\ldots,v_t=g_t$
and we are done.
\end{proof}

\begin{lema}\label{intermed}
Let $n>m\geq 2$ and $t\geq 1$ be some integers with $n\geq 2m+1$. 

Let $V:=(J_{n,m}^t:(x_{n-m+1}x_{n-m+2}\cdots x_{n-1})^t)$. Then:
$$  \sdepth(S/V) \geq \depth(S/V) = \begin{cases} \varphi(n,m,t),& t\leq n-2m \\ 2(m-1),& t\geq n-2m+1 \end{cases}.$$
\end{lema}

\begin{proof}
Let $S':=K[x_1,\ldots,x_{n-m-1}]$ and $S'':=K[x_{n-m},\ldots,x_n]$. We consider the ideals $I=I_{n-m-1,m}\subset S'$
and $L=(x_{n-m},x_n)\subset S''$. According to Lemma \ref{inmt2}, we have that $V=(IS+LS)^t$.
It is clear that 
\begin{equation}\label{dimi}
\dim(S''/L)=m+1-2=m-1. 
\end{equation}
Also, from Theorem \ref{depth}, we have that
\begin{equation}\label{depi}
\sdepth(S'/I^t) \geq \depth(S'/I^t) = \varphi(n-m-1,m,t).
\end{equation}
Since $V=(IS+LS)^t$, from Theorem \ref{teo-iran}, \eqref{dimi}, \eqref{depi} and the fact that $t\mapsto \varphi(n-m-1,m,t)$ is 
nonincreasing, it follows that
$$\sdepth(S/V) \geq \depth(S/V) = \varphi(n-m-1,m,t)+m-1.$$
The required formula follows by straightforward computations.
\end{proof}

The following result is an improvement of Theorem \ref{t212} for $t\leq n-2m$.

\begin{teor}\label{t3}
Let $n > m\geq 2$ and $t\geq 1$ be some integers. If $n\geq 2m+1$ then 
$$\depth(S/J_{n,m}^t) \leq \begin{cases} \varphi(n,m,t),& t\leq n-2m \\ \varphi(n-1,m,t)+1,& n-2m+1\leq t\leq n-m \\ m,& t\geq n-m+1 
\end{cases}.$$
\end{teor}

\begin{proof}
From Lemma \ref{inmt2}, Lemma \ref{intermed} and Lemma \ref{lem}(2) it follows that
\begin{equation}\label{ekku1}
\depth(S/J_{n,m}^t) \leq \begin{cases} \varphi(n,m,t),& t\leq n-2m \\ 2(m-1),& t\geq n-2m+1 \end{cases}.
\end{equation}
On the other hand, according to Theorem \ref{t212} we have that
\begin{equation}\label{ekku2}
\depth(S/J_{n,m}^t) \leq \varphi(n-1,m,t).
\end{equation}
Also, it is easy to check that
$$ \varphi(n-1,m,t)+1\leq 2m-2\text{ for all }t\geq n-2m+1\text{ and } $$
\begin{equation}\label{ekku3}
\varphi(n-1,m,t)=m-1\text{ for all }t\geq n-m+1.
\end{equation}
The required conclusion follows from \eqref{ekku1}, \eqref{ekku2} and \eqref{ekku3}. 
\end{proof}

\begin{obs}\rm
If $m<n\leq 2m$ and $t\geq 1$ then Lemma \ref{inmt2} and Lemma \ref{lem} imply
\begin{align*}
& \depth(S/J_{n,m}^t)\leq \depth(S/(x_{n-m},x_n)^t)=n-2\text{ and }\\
& \sdepth(S/J_{n,m}^t)\leq \sdepth(S/(x_{n-m},x_n)^t)=n-2.
\end{align*}
However, the above inequalities are trivial, since the ideal $J_{n,m}^t$ is not principal.
\end{obs}

\section{Remarks in the general case}\label{s3}

Let $n>m\geq 2$ be two integers.
In \cite{lucrare2} we studied the functions $t\mapsto \depth(S/J_{n,m}^t)$ and
$t\mapsto \sdepth(S/J_{n,m}^t)$ for $t\geq t_0$, where $t_0$ was defined in Section \ref{s2}.
However, for $2\leq t\leq t_0-1$, this problem is much harder. In the following, we 
present a possible way of tackling it and we point out the difficulties which appear.

In order of convenience, we introduce the following notations:
We let $J=J_{n,m}\subset S$, $S'=K[x_1,\ldots,x_{n-1}]$, $J'=(J:x_n)\cap S'$ and $I=I_{n-1,m}\subset S'$.

\begin{lema}\label{inmt}
With the above notations, we have that:
\begin{enumerate}
\item[(1)] $(J^t,x_n^k)=(I^{t+1-k}J^{k-1},x_n^k)$ for all $1\leq k\leq t$.
\item[(2)] $((J^t:x_n^{k-1}),x_n) = (I^{t+1-k}J'^{k-1},x_n)$ for all $1\leq k\leq t$.
\item[(3)] $(J^t,x_n)=(I^t,x_n)$ and $(J^t:x_n^t)=J'^tS$.
\item[(4)] $\depth(S/J^t)\leq \depth(S'/J'^t)+1$ and $\sdepth(S/J^t)\leq \sdepth(S'/J'^t)+1$
\end{enumerate}
\end{lema}

\begin{proof}
(1) Since $IS\subset J$, the inclusion "$\supseteq$" is clear. In order to prove the other inclusion, 
it is enough to note that if $u\in G(J^t)$ such that $x_n^k\nmid u$, then $u\in G(I^{t+1-k}J^{k-1})$.

(2) Since $((J^t:x_n^{k-1}),x_n) = ((J^t,x_n^k):x_n^{k-1})$, the conclusion follows from (1).

(3) It is clear.

(4) From Lemma \ref{lem} and Lemma \ref{lhvz} it follows that
\begin{align*}
& \depth(S/J^t)\leq \depth(S/(J^t:x_n))\leq \cdots \leq \depth(S/(J^t:x_n^t))=\depth(S'/J'^t)+1\text{ and}\\
& \sdepth(S/J^t)\leq \sdepth(S/(J^t:x_n))\leq \cdots \leq \sdepth(S/(J^t:x_n^t))=\sdepth(S'/J'^t)+1.
\end{align*}
Hence, we get the required conclusion.
\end{proof}

With the notations from Proposition \ref{inmt}, let 
$$d_k:=\depth(S'/I^{t+1-k}J'^{k-1})\text{ and }s_k=\sdepth(S'/I^{t+1-k}J'^{k-1})\text{ for }1\leq k\leq t.$$
Note that, according to Theorem \ref{depth}, we have that $s_1\geq d_1 = \varphi(n-1,m,t)$.

\begin{prop}\label{obsy}
We have that:
\begin{enumerate}
\item[(1)] $d_k\geq \depth(S/(J^t:x_n^{k-1}))-1$ for all $1\leq k\leq t$.
\item[(2)] If $\depth(S/(J^t:x_n^{k}))>\depth(S/(J^t:x_n^{k-1}))$ then $d_k = \depth(S/(J^t:x_n^{k-1}))$.
\item[(3)] If $\sdepth(S/(J^t:x_n^{k}))>\sdepth(S/(J^t:x_n^{k-1}))$ then $s_k\leq \sdepth(S/(J^t:x_n^{k-1}))$.
\end{enumerate}
\end{prop}

\begin{proof}
We fix some $k$ with $1\leq k\leq t$.
We have the short exact sequence
\begin{equation}\label{siruri}
 0 \to S/(J^t:x_n^k) \to S/(J^t:x_n^{k-1}) \to S/((J^t:x_n^{k-1}),x_n) \cong S'/I^{t+1-k}J'^{k-1} \to 0,
\end{equation}
where the isomorphism is given by Lemma \ref{inmt}(2).

Thus, from \eqref{siruri}, Lemma \ref{l11} (Depth lemma) and Lemma \ref{asia} it follows that 
\begin{align*}
& \depth(S/(J^t:x_n^{k-1})) \geq \min\{\depth(S/(J^t:x_n^{k})) , d_k \}, \\
& d_k \geq \min\{\depth(S/(J^t:x_n^{k}))-1,\depth(S/(J^t:x_n^{k-1})) \}\text{ and }\\
& \sdepth(S/(J^t:x_n^{k-1})) \geq \min\{\sdepth(S/(J^t:x_n^{k})) , s_k \}.
\end{align*}
Now, we get the required conclusions (1-3).
\end{proof}

\begin{prop}\label{obsy2}
With the above notations, we have that:
\begin{enumerate}
\item[(1)] $\depth(S/(J^t,x_n^t))\geq \min\{\varphi(n-1,m,t),d_2,\ldots,d_t\}$.
\item[(2)] $\sdepth(S/(J^t,x_n^t))\geq \min\{\varphi(n-1,m,t),s_2,\ldots,s_t\}$.
\item[(3)] $\depth(S/J^t)\leq \depth(S/(J^t,x_n^t))+1$.
\item[(4)] If $\depth(S/(J^t:x_n)^t)>\depth(S/J^t)$ then $\depth(S/J^t)\geq \depth(S/(J^t,x_n^t))$ and
           similarly for the Stanley depth.
\end{enumerate}
\end{prop}

\begin{proof}
(1) We consider the following short exact sequences
\begin{equation}\label{siruri2}
 0 \to S/((J^t:x_n^{k-1}),x_n) \to S/(J^t,x_n^{k}) \to S/(J^t,x_n^{k-1}) \to 0\text{ for }2\leq k\leq t.
\end{equation}
Since, by Lemma \ref{inmt}, we have that $S'/I^{t+1-k}J'^{k-1}\cong S/((J^t:x_n^{k-1}),x_n)$,
the conclusion follows from \eqref{siruri2}, Lemma \ref{l11} (Depth lemma) and Theorem \ref{depth}.

(2) The proof is similar to the proof of (1), using Lemma \ref{asia} instead of Lemma \ref{l11}.

(3) It follows from Lemma \ref{l11}, Lemma \ref{lem} and 
the short exact sequence 
\begin{equation}\label{copacel}
0\to S/(J^t:x_n^t) \to S/J^t \to S/(J^t,x_n^t) \to 0.
\end{equation}
(4) It follows form \eqref{copacel}, Lemma \ref{l11} and Lemma \ref{asia}.
\end{proof}

The following examples illustrate the computational difficulties which appear, when we apply 
the previous results in order to compute $\depth(S/J^t)$ and $\sdepth(S/J^t)$.

\begin{exm}\label{exem1}\rm
Let $J=J_{6,3}^2\subset S=K[x_1,\ldots,x_6]$. Let $I=I_{5,3}^2\subset S'=K[x_1,\ldots,x_5]$.
We have $$J'=(J:x_6)\cap S'=(x_1x_2,x_2x_3x_4,x_4x_5,x_5x_1).$$
As in the proof of Proposition \ref{obsy}, we have the short exact sequences:
\begin{align}\label{shorte}
& 0 \to S/(J^2:x_6) \to S/J^2 \to S/(J^2,x_6)\cong S'/I^2 \to 0\text{ and }\\
& 0 \to S/(J^2:x_6^2) \to S/(J^2:x_6) \to S/((J^2:x_6),x_6) \cong S'/IJ' \to 0.
\end{align}
From Theorem \ref{depth} it follows that:
$$\sdepth(S'/I^2)\geq\depth(S'/I^2)=\varphi(5,3,2)=3.$$
Note that $IJ'=x_3L$, where $L=(x_1x_2,x_2x_4,x_4x_5)(x_1x_2,x_2x_3x_4,x_4x_5,x_5x_1)\subset S'$, and
$(J^2:x_6^2)=J'^2S=(x_1x_2,x_2x_3x_4,x_4x_5,x_5x_1)^2S$. From Lemma \ref{lemm}, it follows that
\begin{equation}\label{sprim}
\depth(S'/IJ')=\depth(S'/L)\text{ and }\sdepth(S'/IJ')=\sdepth(S'/L).
\end{equation}
We have $(L:x_2x_3x_4)=(x_1,x_4)\cap(x_2,x_5)$. By straightforward computations we get:
$$\sdepth(S'/(L:x_2x_3x_4))=\depth(S'/(L:x_2x_3x_4))=2.$$
Also, $W:=(L,x_2x_3x_4)=(x_2x_3x_4, x_1x_2x_4x_5, x_1^2x_2x_5, x_1^2x_2^2, x_2x_4^2x_5, x_1x_2^2x_4, x_4^2x_5^2, x_1x_4x_5^2)$.
We have that $(W:x_2x_4x_5)=(x_3,x_4,x_1)$. Therefore we get 
$$\sdepth(S'/(W:x_2x_4x_5))=\depth(S'/(W:x_2x_4x_5))=2.$$
Also, $(W,x_4x_5)=(x_2x_4x_5, x_2x_3x_4, x_1x_4x_5^2, x_4^2x_5^2, x_1x_2^2x_4, x_1^2x_2^2, x_1^2x_2x_5)$.
By continuing the computations, we can deduce that
$$\sdepth(S'/(W,x_4x_5))=\depth(S'/(W,x_4x_5))=2.$$
From the short exact sequence $0\to S'/(W:x_4x_5)\to S'/W \to S'/(W,x_4x_5)\to 0$, using Lemma \ref{l11} and Lemma \ref{asia},
we deduce that:
$$\depth(S'/W)\geq 2\text{ and }\sdepth(S'/W)\geq 2.$$
From the short exact sequence $0\to S'/(L:x_2x_3x_4)\to S'/L \to S'/W \to 0$ and \eqref{sprim}, using Lemma \ref{l11} and Lemma \ref{asia},
we deduce that: 
$$d_1:=\depth(S'/IJ')=\depth(S'/L)\geq 2\text{ and }s_1:=\sdepth(S'/IJ')=\sdepth(S'/L)\geq 2.$$
Using similar computations, as above, one can deduce that
$$d_2=\depth(S'/J'^2)\geq 2\text{ and }s_2:=\sdepth(S'/IJ')=\sdepth(S'/L)\geq 2.$$

We have $(L,x_4)=(x_1^2x_2(x_1,x_5),x_4)$ and therefore it is easy to see that
$$\sdepth(S'/(L,x_4))=\depth(S'/(L,x_4))=2.$$
Let $K:=(L:x_4)$.
We have $(K:x_3)=(x_1x_2^2,x_1x_2x_5,x_1x_5^2,x_2^2x_4,x_2x_4x_5,x_4x_5^2)$ and
$(K,x_3)=(x_3, x_1x_2^2,x_1x_2x_5,x_1x_5^2,x_2x_4x_5,x_4x_5^2)$.
By continuing the computations, we can deduce that 
$$\sdepth(S'/(K:x_3))=\depth(S'/(K:x_3))=2\text{ and }\sdepth(S'/(K,x_3))=\depth(S'/(K,x_3))=1.$$
From the short exact sequence $0\to S'/(K:x_3) \to S'/K \to S'/(K,x_3) \to 0$ it follows that
$$\depth(S'/K)\geq 1\text{ and }\sdepth(S'/K)\geq 1.$$
From the short exact sequence $0\to S'/(L:x_4)=S'/K \to S'/L \to S'/(L,x_4)\to 0$ it follows that
$$\depth(S'/IJ')=\depth(S'/L)=\depth(S'/K)\geq 1\text{ and }\sdepth(S'/IJ')=\sdepth(S'/L)\geq 1.$$
As in the proof of Proposition \ref{obsy}, we can deduce from the short exact sequences \eqref{shorte} that:
$$\depth(S/J^2)\geq 2\text{ and }\sdepth(S/J^2)\geq 2.$$
We mention that, according to Cocoa \cite{cocoa}, we have $\sdepth(S/J^2)=\depth(S/J^2)=3$.
\end{exm}

\begin{exm}\label{exem2}\rm
Let $J=J_{6,4}\subset S=K[x_1,\ldots,x_6]$. Let $I=I_{5,4}\subset S'=K[x_1,\ldots,x_5]$.
We have $$J'=(J:x_6)\cap S'=(x_1x_2x_3,x_3x_4x_5,x_4x_5x_1,x_5x_1x_2).$$
We have the short exact sequences
\begin{align}\label{shortes}
& 0 \to S/(J^2:x_6) \to S/J^2 \to S/(J^2,x_6)\cong S'/I^2 \to 0\text{ and }\\
& 0 \to S/(J^2:x_6^2) \to S/(J^2:x_6) \to S/((J^2:x_6),x_6) \cong S'/IJ' \to 0.
\end{align}
From Theorem \ref{depth} it follows that
$$\sdepth(S'/I^2)\geq\depth(S'/I^2)=\varphi(5,4,2)=3.$$
Note that $IJ'=x_2x_3x_4L$, where $L=(x_1,x_5)(x_1x_2x_3,x_3x_4x_5,x_4x_5x_1,x_5x_1x_2)\subset S'$.
From Lemma \ref{lemm}, it follows that
\begin{equation}\label{sprims}
\depth(S'/IJ')=\depth(S'/L)\text{ and }\sdepth(S'/IJ')=\sdepth(S'/L).
\end{equation}
We  have $(L:x_3)=(x_1,x_5)(x_1x_2,x_4x_5) \text{ and }(L,x_3)=x_1x_5(x_1,x_5)(x_2,x_4)+(x_3)$.
By straightforward computation we have
\begin{align*}
& \sdepth\left(\frac{K[x_1,x_2,x_4,x_5]}{(x_1,x_5)(x_2,x_4)}\right)=\depth\left(\frac{K[x_1,x_2,x_4,x_5]}{(x_1,x_5)(x_2,x_4)}\right)=1 \\
& \sdepth\left(\frac{K[x_1,x_2,x_4,x_5]}{(x_1,x_5)(x_1x_2,x_4x_5)}\right)=\depth\left(\frac{K[x_1,x_2,x_4,x_5]}{(x_1,x_5)(x_1x_2,x_4x_5)}\right)=2.
\end{align*}
Therefore, using Lemma \ref{lemm}, we obtain $\sdepth(S/(L,x_3))=\depth(S/(L,x_3))=1$. Also $\sdepth(S/(L:x_3))=\depth(S/(L:x_3))=3$.
From the short exact sequence $$0\to S/(L:x_3)\to S/L \to S/(L,x_3) \to 0$$ we deduce that $\depth(S/L),\sdepth(S/L)\geq 1.$

Hence, from \eqref{sprims}
it follows that $\depth(S'/IJ'),\sdepth(S'/IJ')\geq 1$. Using similar computations, one can deduce that $\depth(S'/J'^2),\sdepth(S'/J'^2)\geq 2$.

Therefore, from \eqref{shortes}, we deduce that $\depth(S/J^2),\sdepth(S/J^2)\geq 1$. Note that 
$$(J^2:x_1x_2x_3x_4x_5x_6)=(x_1,x_3,x_5)\cap(x_2,x_4,x_6).$$
Therefore, according to Lemma \ref{liema}, it follows that $\depth(S/(J^2:x_1x_2x_3x_4x_5x_6))=1$.
Consequently, we get $\depth(S/J^2)=1$.
\end{exm}

\subsection*{Acknowledgements}

We gratefully acknowledge the use of the computer algebra system Cocoa (cf. \cite{cocoa}) for our experiments.

The second author was supported by a grant of the Ministry of Research, Innovation and Digitization, CNCS - UEFISCDI, 
project number PN-III-P1-1.1-TE-2021-1633, within PNCDI III.


\begin{thebibliography}{99}

			
\bibitem{lucrare1} S.\ B\u al\u anescu, M.\ Cimpoea\c s, \textit{Depth and Stanley depth of powers of the path ideal of a path graph}, 
      https://arxiv.org/pdf/2303.01132.pdf (2023).

\bibitem{lucrare2} S.\ B\u al\u anescu, M.\ Cimpoea\c s, \textit{Depth and Stanley depth of powers of the path ideal of a cycle graph}, 
      https://arxiv.org/pdf/2303.15032.pdf (2024).
			
\bibitem{mir} M.\ Cimpoea\c s, \textit{Stanley depth of monomial ideals with small number of generators},
      Central European Journal of Mathematics, \textbf{vol. 7, no. 4}, (2009), 629--634.

\bibitem{mirci} M.\ Cimpoea\c s, \textit{Several inequalities regarding Stanley depth}, Romanian Journal of Math. and Computer Science 
\textbf{2(1)}, (2012), 28--40.




\bibitem{iran} M.\ Cimpoea\c s, \textit{On the Stanley depth of powers of some classes of monomial ideals},
               Bull. Iranian Math. Soc. \textbf{Vol. 44 , no. 3}, (2018), 739--747.


\bibitem{conca} A. Conca, E. De Negri, \emph{M-sequences, graph ideals and ladder ideals of linear types}, J. Algebra \emph{211} (1999),
599--624.

\bibitem{cocoa} CoCoATeam, \textit{CoCoA: a system for doing Computations in Commutative Algebra}, Available at http://cocoa.dima.unige.it

\bibitem{duval} A.\ M.\ Duval, B.\ Goeckneker, C.\ J.\ Klivans, J.\ L.\ Martine, \emph{A non-partitionable Cohen-Macaulay simplicial complex}, Advances in Mathematics \textbf{299} (2016), 381--395.

\bibitem{hvz} J.\ Herzog, M.\ Vl\u adoiu, X.\ Zheng, \textit{How to compute the Stanley depth of a monomial ideal},
           Journal of Algebra \textbf{322(9)}, (2009), 3151--3169.
           



\bibitem{asia} A.\ Rauf, \textit{Depth and sdepth of multigraded module}, Communications in Algebra,
\textbf{Vol. 38, Issue 2}, (2010), 773--784.


\bibitem{stan} R.\ P.\ Stanley, \emph{Linear Diophantine equations and local cohomology}, Invent. Math. \textbf{68}, (1982), 175--193.    


\bibitem{real} R.\ H.\ Villarreal, \emph{Monomial algebras. Second edition}, Monographs and Textbooks in Pure and
Applied Mathematics, Chapman \& Hall, New York, 2018.


\end{thebibliography}
\end{document}